\documentclass[12pt]{amsart}
\usepackage{graphicx}

\theoremstyle{theorem}%
\newtheorem{theorem}{Theorem}[section]

\newtheorem{lemma}[theorem]{Lemma}
\theoremstyle{definition}%

\begin{document}

\title[Hamiltonian Cycles in Multisigned Graphs]{Hamiltonian Cycles in Signed and Multisigned Complete Graphs}

\author{Xiyong Yan}

\address{89 Park Ave, Apt 29,  Binghamton, NY,  USA,  13903.}

\begin{abstract}
A signed complete graph on $n$ vertices contains both positive and negative Hamiltonian cycles if and only if it also contains both positive and negative triangles. Otherwise, all Hamiltonian cycles are negative if and only if all triangles are negative and $n$ is odd, while all Hamiltonian cycles are positive if and only if all triangles are negative and $n$ is even, or all triangles are positive. Extending these results to multisigned complete graphs, we prove that such a graph contains at least two Hamiltonian cycles with different multisigns if and only if it contains at least two triangles with different multisigns.
\end{abstract}

\maketitle

\section{Introduction}
 Our research question originates from the paper ``Negative (and Positive) Circles in Signed Graphs: A Problem Collection'' \cite{2}. In it, the author poses the question: Assume a signed graph has a Hamiltonian cycle and is unbalanced. Is there a negative Hamiltonian cycle? A positive one? In our study, we narrow the focus to signed complete graphs and expand it to multisigned complete graphs. 

We establish the relationship between the multisigns of Hamiltonian cycles and the multisigns of triangles in a multisigned complete graph. To achieve this, we utilize the Hourglass Lemma and the Consistent Triangle Multisign Lemma. In the special case of signs, this shows the complete relationship between the signs of Hamiltonian cycles and the signs of triangles in a signed complete graph.

\section{Preliminaries}
 A \textbf{signed graph} $\Sigma=(\Gamma,\sigma)$ consists of an underlying graph $\Gamma=(V, E)$, where $V$ denotes the set of vertices and $E$ the set of edges, and a sign function $\sigma$ that assigns either $+$ or $-$ to each edge in $\Gamma$. 
 
A \textbf{multisigned graph} \( \Sigma = (\Gamma, \sigma, G) \) consists of the following components: 
The base graph is \( \Gamma = (V, E) \). 
The group $G:=\{-1,+1\}^m$ with componentwise multiplication. 
 The identity element of $G$ is denoted by $e$.
The multisign function is \( \sigma: E \to G \) . 

The \textbf{multisign} of a multisigned cycle $P$, denoted by $\sigma(P)$,  is defined as the product of the multisigns of all edges of $P$.   $\Sigma$ is called \textbf{balanced} if every cycle has multisign $e$.

\section{Hourglass Lemma: Sign Consistency of Hamiltonian Cycles}
\begin{lemma} [Hourglass Lemma]\label{hour}
    Let $v_1,v_2,\dots,v_n$ be vertices of a multisigned complete graph $\Sigma= (K_n, \sigma, G)$. Two  Hamiltonian cycles $v_1v_2\dots v_n$ and $v_1v_2\dots v_iv_jv_{j-1}\dots v_{i+1}v_{j+1}\dots v_n$, have the same multisign if and only if  $\sigma(v_iv_{i+1})\sigma(v_jv_{j+1})=\sigma(v_iv_j)\sigma(v_{i+1}v_{j+1})$, where $j>i$ and $v_i,v_{i+1},v_j,v_{j+1}$ are distinct for some $i,j\in\{1,2,\dots,n\}$.
\end{lemma}

\begin{figure}[h]
    \includegraphics[width=0.8\linewidth]{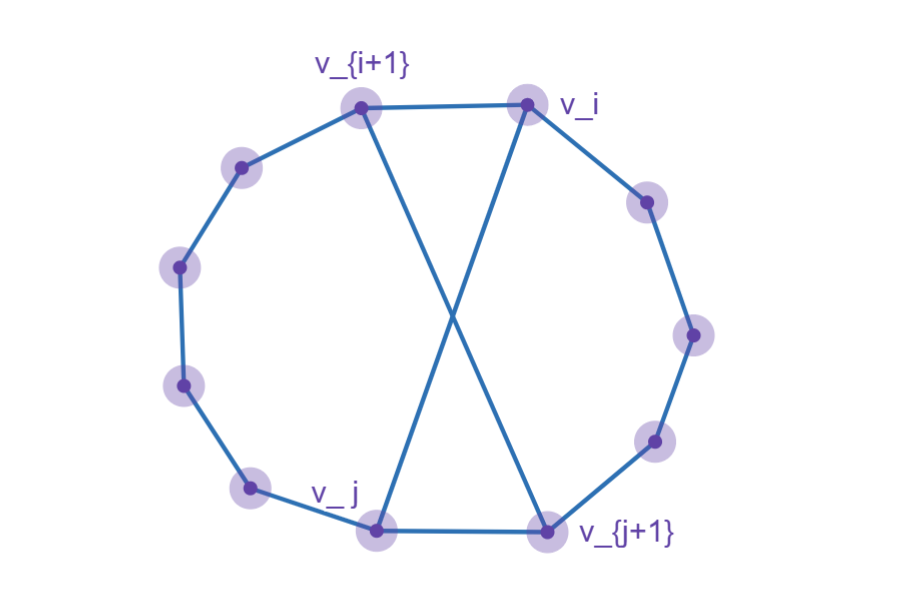}
    \caption{Illustration of the Hourglass Shape used in the proof of Lemma \ref{hour}.}
\end{figure}

\begin{proof}  Consider two  Hamiltonian cycles $v_1v_2\dots v_iv_jv_{j-1}\dots v_{i+1}v_{j+1}\dots v_n$ and $v_1v_2\dots v_n$, where $j>i$ and $v_i,v_{i+1},v_j,v_{j+1}$ are distinct. Let
$$ A=\sigma(v_1v_2\dots v_n)\sigma(v_1v_2\dots v_iv_jv_{j-1}\dots v_{i+1}v_{j+1}\dots .v_n).$$
Break each multisign of the Hamiltonian cycle into the product of the multisigns of its edges. We get
\begin{align*}
    A=&\sigma(v_1v_2)\sigma(v_2v_3)\dots\sigma(v_nv_1) \cdot \sigma(v_1v_2)\sigma(v_2v_3)\dots\sigma(v_{i-1}v_i)\sigma(v_iv_j) \\
    &\cdot \sigma(v_jv_{j-1})\dots\sigma(v_{i+2}v_{i+1})\sigma(v_{i+1}v_{j+1})\sigma(v_{j+1}v_{j+2})\dots\sigma(v_nv_1) .
\end{align*}
That is, $$A=\sigma(v_iv_{i+1})\sigma(v_jv_{j+1})\sigma(v_iv_j)\sigma(v_{i+1}v_{j+1}).$$
Therefore, if the two Hamiltonian cycles have the same multisign, 
$A$ has a multisign $e$. This implies
$$\sigma(v_iv_{i+1})\sigma(v_jv_{j+1})=\sigma(v_iv_j)\sigma(v_{i+1}v_{j+1}).$$
Conversely, if the two Hamiltonian cycles have different multisigns, the above equation does not hold. As required. 
\end{proof}

\section{Characterization of Unbalanced Complete Graphs and Consistent Triangle Multisign Lemma}

\begin{lemma}\label{thm2} 
Let $P$ be a cycle in a multisigned complete graph with vertices $v_1,v_2,\dots,v_k$. After adding edges $v_1v_3,v_1v_4,\dots,v_1v_{k-1}$ to $P$, decompose the resulting graph into triangles $\triangle v_1v_2v_{3},\triangle v_1v_3v_{4},\dots,\triangle v_1v_{k-1}v_{k}.$ Then the multisign of $P$ is equal to the product of the multisigns of all these triangles.
\end{lemma}

\begin{lemma}\label{thm4}  
A multisigned complete graph $\Sigma$ is unbalanced if and only if $\Sigma$ contains a triangle with a non-identity multisign.
\end{lemma}
\begin{proof} Consider a multisigned complete graph $\Sigma$ that is unbalanced. Assume, for contradiction, that all triangles have multisign $e$. By Lemma \ref{thm2}, all cycles in $\Sigma$ would then have multisign $e$. This implies that $\Sigma$ is balanced, which contradicts the assumption that $\Sigma$ is unbalanced. Thus, $\Sigma$ contains a triangle with a non-identity multisign.

Conversely, suppose $\Sigma$ contains a  triangle with a non-identity multisign. Then by definition, $\Sigma$ is unbalanced. 
\end{proof}

\begin{lemma}\label{cor1}
    If all the Hamiltonian cycles have the same multisign in a multisigned complete graph $\Sigma$, where $n \geq 4$ , then all triangles  have the same multisign.
\end{lemma}
\begin{figure}[h]
    \centering
    \includegraphics[width=0.8\linewidth]{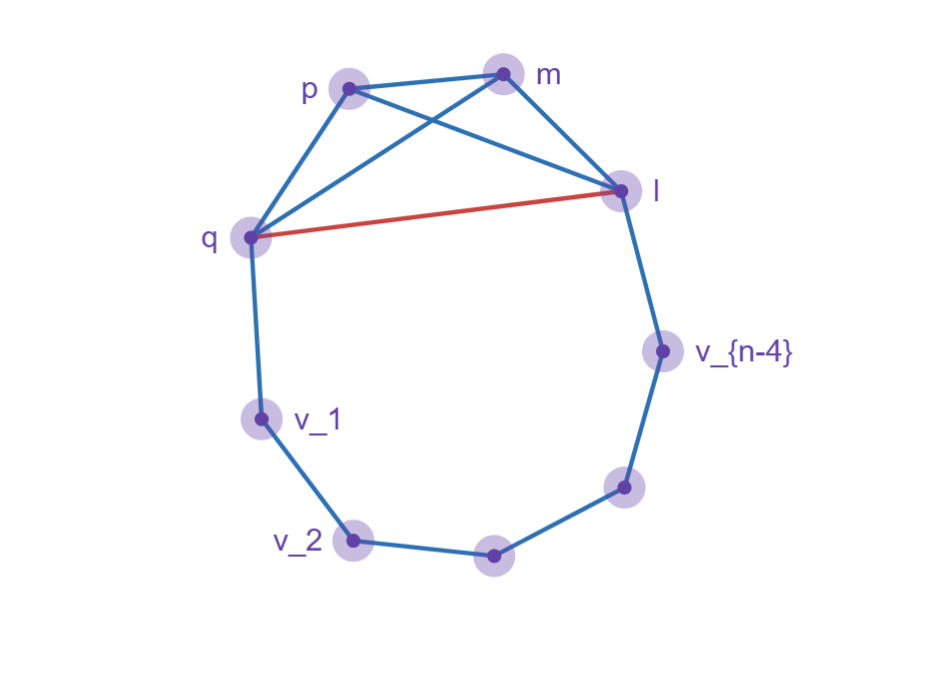}
    \caption{Illustration for the proof of Lemma \ref{cor1}.}
    \label{fig:enter-label}
\end{figure}
\begin{proof}
Let the vertices of $\Sigma$ be labeled $1,2,\dots,n$. Let $\triangle lmq$ and $\triangle lpq$ be two triangles that share an edge $lq$. 
Consider two Hamiltonian cycles 
$lmpqv_1v_2\dots v_{n-4}$ and $lpmqv_1v_2\dots v_{n-4}$. Since we assume that all the Hamiltonian cycles in $\Sigma$ have the same multisign, by the Hourglass Lemma we have
$$\sigma(pq)\sigma(l m)=\sigma(lp)\sigma(mq).$$
This implies 
$$\sigma(l m)\sigma(mq)\sigma(l q)=\sigma(pq)\sigma(lp)\sigma(lq).$$
That is equivalent to $$\sigma(\triangle l mq)=\sigma(\triangle lpq).$$
Hence, these two triangles have the same multisign.

Let $\triangle v_iv_jv_k$ and $\triangle v_rv_sv_t$ be any two triangles in $\Sigma$, for distinct $i,j,k,r,s,t\in \{1,2,\dots ,n\} $. By the result from last paragraph, we have
$$
    \sigma(\triangle v_iv_jv_k)=\sigma(\triangle v_iv_jv_r)=\sigma(\triangle v_iv_rv_s)=\sigma(\triangle v_rv_sv_t).
$$
If two indices are equal, e.g., $i=t$, the argument is the same.  This implies that all the triangles in $\Sigma$ have the same multisign. As required.
\end{proof}

\section{Signs of Hamiltonian Cycles}
The relationship between the signs of Hamiltonian cycles and the signs of triangles is as follows:
\begin{theorem}\label{thm} 
Consider a signed complete  graph $\Sigma_{n,\sigma} = (K_n, \sigma)$, where $n\geq 4$.

(a) All Hamiltonian cycles are negative if and only if all triangles are negative and $n$ is odd, while all Hamiltonian cycles are positive if and only if all triangles are negative and $n$ is even, or all triangles are positive.

(b)  A signed complete graph $\Sigma_{n,\sigma}$ contains both a positive Hamiltonian cycle and a negative Hamiltonian cycle if and only if it contains at least one positive triangle and at least one negative triangle.
\end{theorem}

\begin{proof}
We deduce Theorem \ref{thm} from Theorem \ref{mul} by taking 
$m=1$. 
\end{proof}

\begin{theorem}\label{mul}
 Consider a multisigned complete  graph $\Sigma = (K_n, \sigma,G)$, where $n\geq 4$ and $m\geq 1$.   
 
 (a)  Let $g$ be a fixed non-identity element of the multisign group $G$. All Hamiltonian cycles in $\Sigma$ have the same multisign $g$ if and only if all triangles have multisign $g$ and $n$ is odd. Similarly, all Hamiltonian cycles have multisign $e$ if and only if either all triangles have the same non-identity multisign and $n$ is even, or all triangles have multisign $e$. 
 
 (b) The graph $\Sigma$ contains at least two  Hamiltonian cycles with different multisigns if and only if it contains at least two triangles with different multisigns.
 
\end{theorem}

\begin{proof}
Proof of part (a).
Let $P$ be any Hamiltonian cycle in $\Sigma$ with vertices $v_1,v_2,...,v_n$. Construct triangles as in Lemma \ref{thm2}. By Lemma \ref{thm2}, we have $$\sigma(P)=\sigma(\triangle v_1v_2v_3)\sigma(\triangle v_1v_3v_4)\dots\sigma(\triangle v_1v_{n-1}v_n).$$  

 If all triangles have the same non-identity multisign $g$ and $n$ is odd, then $\sigma(P)=g^{n-2}=g.$ Thus, all Hamiltonian cycles have multisign $g$. Conversely, assume all Hamiltonian cycles have the same non-identity multisign $g$. By Lemma \ref{cor1}, all triangles have the same multisign. Since the graph is unbalanced, it contains a  triangle with a non-identity multisign $h\in G$ by Lemma  \ref{thm4}, which implies that all triangles have the same multisign $h$.  Let $P$ be any Hamiltonian cycle. From Lemma \ref{thm2}, we have $\sigma(P)=h^{n-2}$. If $n$ is even, this implies $\sigma(P)=e$, which is a contradiction. Thus, $n$ is odd. Now, all Hamiltonian cycles have multisign $g$, so $g=\sigma(P)=h^{n-2} = h$.  Thus, all triangles have multisign $g$.

If all triangles have the same multisign $h$ and either $n$ is even or $h = e$, then $\sigma(P)=h^{n-2} =e.$ Thus, all Hamiltonian cycles have multisign $e$. 
 
Conversely, assume that all Hamiltonian cycles have multisign $e$.  By Lemma \ref{cor1}, all triangles have the same multisign $h$.  
If $\Sigma$ is unbalanced, using reasoning similar to that in the previous part, we conclude that $n$ must be even and $e=\sigma(P)=h^{n-2}$, for some non-identity $h\in G$. This implies that all triangles have the same multisign $h$. On the other hand, if $\Sigma$ is balanced, all triangles have multisign $e$.

Part (b) follows directly from part (a). 
\end{proof}

\section{Acknowledgment}

The author thanks Professor Thomas Zaslavsky for his invaluable suggestions and recommendations during the preparation of this draft.


\end{document}